\documentclass[11pt,reqno]{amsart}
\usepackage{fullpage}
\usepackage{url}

\usepackage{amsmath}
\usepackage{amsfonts}
\usepackage{amssymb}
\usepackage{amsthm}
\usepackage{mathtools}
\usepackage{mathrsfs}
\usepackage{float}
\usepackage[all,cmtip]{xy}

\usepackage{color}	
\definecolor{due}{RGB}{0,76,147}
      \theoremstyle{definition}
\newtheorem{defi}{Definition}[section]
\theoremstyle{plain}
\newtheorem{thm}[defi]{Theorem}
\newtheorem{conj}[defi]{Conjecture}

\newtheorem{cor}[defi]{Corollary}
\newtheorem{lemma}[defi]{Lemma}
\theoremstyle{remark}

\newtheorem{rmk}[defi]{Remark}
\newtheorem{ex}[defi]{Example}

\theoremstyle{definition}

\newcommand{\longra}{\longrightarrow}

\newcommand{\kl}{{\mathcal L}}
\newcommand{\km}{{\mathcal M}}
\newcommand{\kn}{{\mathcal N}}

\newcommand{\kp}{{\mathcal P}}

\newcommand{\C}{{\mathbb C}}

\newcommand{\IP}{{\mathbb P}}
\newcommand{\Q}{{\mathbb Q}}

\newcommand{\Sing}{\operatorname{Sing}}

  \makeatletter
\newcommand{\xdashrightarrow}[2][]{\ext@arrow 0359\rightarrowfill@@{#1}{#2}}
\newcommand{\xdashleftarrow}[2][]{\ext@arrow 3095\leftarrowfill@@{#1}{#2}}
\newcommand{\xdashleftrightarrow}[2][]{\ext@arrow 3359\leftrightarrowfill@@{#1}{#2}}
\def\rightarrowfill@@{\arrowfill@@\relax\relbar\rightarrow}
\def\leftarrowfill@@{\arrowfill@@\leftarrow\relbar\relax}
\def\leftrightarrowfill@@{\arrowfill@@\leftarrow\relbar\rightarrow}
\def\arrowfill@@#1#2#3#4{%
  $\m@th\thickmuskip0mu\medmuskip\thickmuskip\thinmuskip\thickmuskip
   \relax#4#1
   \xleaders\hbox{$#4#2$}\hfill
   #3$%
}
\makeatother

\begin{document}
\title{Local negativity of surfaces with non-negative Kodaira dimension and transversal configurations of curves}
\author{Roberto Laface \and Piotr Pokora}
\date{\today}


\address{Roberto Laface \newline Technische Universit\"at M\"unchen, Zentrum Mathematik - M11,
    Boltzmannstra{\ss}e 3,
    85748 Garching bei M\"unchen, Germany
}
    \email{laface@ma.tum.de}

\address{Piotr Pokora \newline Institut f\"ur Algebraische Geometrie,
    Leibniz Universit\"at Hannover,
    Welfengarten 1,
    D-30167 Hannover, Germany}
   \curraddr{Institute of Mathematics Polish Academy of Sciences, \'Sniadeckich 8, PL-00-656 Warsaw, Poland}
   \email{piotrpkr@gmail.com}

\maketitle

\thispagestyle{empty}
\begin{abstract}
We give a bound on the H-constants of configurations of smooth curves having transversal intersection points only on an algebraic surface of non-negative Kodaira dimension. We also study in detail configurations of lines on smooth complete intersections $X \subset \IP^{n+2}_\C$ of multi-degree $d=(d_1, \dots, d_n)$, and we provide a sharp and uniform bound on their H-constants, which only depends on $d$. 
\end{abstract}


\section{Introduction}

In this article, we carry on with our study of the local negativity phenomenon for algebraic surfaces. This is strictly related to the following celebrated conjecture, which dates back to the beginning of XX century.

\begin{conj}[Bounded Negativity Conjecture (BNC)]
Let $X$ be a smooth projective surface defined over a field of characteristic zero. Then there exists an integer $b(X) \in \mathbb{Z}$ such that for all \emph{reduced} curves $C \subset X$ one has $C^{2} \geq -b(X)$.
\end{conj}

This conjecture is known to hold for some classes of surfaces, for instance surfaces with $\Q$-effective anti-canonical divisor (it follows from the adjuction formula), or for surfaces equipped with a surjective endomorphism which is not an isomorphism (see \cite[Proposition 2.1]{Duke}). However, even in those cases, the conjecture is widely open if we start blow up points. As an example, if we consider the blowup $X_{s}$ of $\mathbb{P}^2_\C$ at $s\geq 10$ very general points, then it is not known whether it has bounded negativity. We predict that in that case one should have $b(X_{s}) = s$, for further details please consult \cite{Duke, Harbourne1}.

In \cite{BdRHHLPSz}, the authors have introduced the notion of H-constant\footnote{In fact, they defined the so-called Hadean constant, although this notion should be attributed to Brian Harbourne.}, which allows one to (potentially) study the bounded negativity property for all blow-ups of a given surface simultaneously. The H-constant should also be thought of as an asymptotic version of the self-intersection numbers, and it provides a more effective approach to the study on the BNC. Let $X$ be a smooth projective surface, and consider a non-empty collection $\mathcal{P}=\lbrace P_1, \dots , P_s \rbrace$ of mutually distinct points on $X$. Then, the \emph{H-constant at $\mathcal{P}$} is defined to be
\[ H(X; \mathcal{P}) := \inf_{C} \frac{\tilde{C}^2}{s},\]
where $\tilde{C}$ is the strict transform of $C$ along the blow-up  $\sigma: \tilde{X} \longra X$ of $X$ at the points of $\kp$, and the infimum is taken over all reduced (possibly reducible) curves on $X$. The \emph{global H-constant of $X$} is the quantity
\[ H(X) := \inf_\mathcal{P} H(X; \mathcal{P}),\]
where the infimum is taken over all non-empty collections of mutually distinct points on $X$. 

The importance of the notion of H-constants in the study of the BNC is highlighted in the following remark: if $H(X) > -\infty$, then, for any collection $\mathcal{P}$ of mutually distinct points on $X$, the BNC holds on the blow-up of $X$ at the points of $\mathcal{P}$. Nevertheless, even if $H(X) = -\infty$, BNC might still be true on $X$ or any of its blow-ups at mutually distinct points.

In this note, we consider the notion of {local H-constant} of a configuration of smooth curves having pairwise transversal intersection points only. This is a quite natural variation of the original H-constants: if $C= C_1 + \cdots + C_n$ is a configuration of smooth curves with $n\geq 2$ having pairwise transversal intersection on $X$, the rational number
   \begin{equation}\label{eq:TA Harbourne constant}
      h(X;{C}) := \frac{C^2-\sum_{P \in {\rm Sing}(C)}({\rm mult}_P(C))^2}{s}
   \end{equation}
   is the \emph{local H-constant of the transversal configuration} ${C} \subset X$, where $s$ is equal to the cardinality of ${\rm Sing}(C)$. Let us emphasize here that in this note, unless otherwise specified, by H-constants we will always mean local H-constants. 
Here is the structure of our paper. Section \ref{1} is concerned with configurations of smooth curves having transversal intersection points on algebraic surfaces: generalizing the techniques in \cite{RP2015}, we were able to extend our previous results to arbitrary surfaces of non-negative Kodaira dimension (Theorems \ref{h-constant} and \ref{bound}). Similar results have been obtained in \cite{PSR}. However, the techniques used therein are very much different. In particular, using the logarithmic Miyaoka-Yau inequality \cite{Miyaoka1} one gets a clear bound on the H-constants, and it also allows to deduce immediately some uniform bounds for certain classes of curves, namely rational and elliptic curves.

In Section \ref{3}, we study configurations of lines on smooth surfaces that are complete intersections $X \subset \IP^{n+2}_\C$. In \cite[Theorem 2.2]{RP2015}, we showed that the obtained estimate on H-constants is quite accurate for configurations of rational curves on K3 and Enriques surface, and we provided a variety of examples. However, the accuracy of this bound is not quite as good, for instance, in the case of lines on a smooth hypersurface $X_{d} \subset \IP^3_\C$ of degree $d\geq 4$. We were able to prove (Theorem \ref{degree_d}) that, for a connected configuration of lines on such a surface, H-constants are bounded by $-d(d-1)$ (cf. \cite[Main Result]{Pokora}), and that this bound is sharp. Similarly, in the case of non-connected configurations, H-constants are bounded by a polynomial of degree 3 in $d$. In fact, building on the hypersurfaces case, one obtains similar bound for configurations of lines on a smooth complete intersection $X=(d_1, \dots, d_n) \subset \IP^{n+2}_\C$, as we remark in Theorem \ref{complete_intersection}.

\section{H-constants for surfaces with non-negative Kodaira dimension}\label{1}

Let $C$ be a configuration of (distinct) smooth curves on a surface $X$ of non-negative Kodaira dimension. The configuration $C$ is \emph{transversal} if the irreducible components of $C$ meet pairwise transversally; otherwise said, locally at every singular point of $C$ (seen as a curve), $C$ is analytically equivalent to a configuration of lines meeting at a single point.

In a previous paper \cite{RP2015}, we have studied the local negativity for configurations of smooth rational curves on surfaces with numerically trivial canonical class. Let us recall that $P$ is an \emph{r-fold point} (or an \emph{r-point}) of the configuration ${C}$
   if it is contained in exactly $r$ irreducible components of ${C}$. The union of all $r$-fold points
   $P\in {C}$ for $r\geq 2$, is the singular set $\Sing({C})$ of ${C}$.
   We set the number $t_r=t_r({C})$ to be the number of $r$-fold points in ${C}$. We can rephrase H-constants by using the $t_{i}$'s. More precisely, if $C$ is a transversal configuration of curves and $s$ is equal to the number of singular points of $C$, then 
$$h(X; {C}) = \frac{{C}^2-\sum_{r\geq 2}r^{2}t_{r}}{s}.$$

In \cite[Theorem 2.2]{RP2015}, we considered the case of surfaces with numerically trivial canonical class. The following result is a generalization of this result to the case of surfaces of non-negative Kodaira dimension, and it is inspired by an idea of Miyaoka \cite[Section~2.4]{Miyaoka1}.

\begin{thm}
\label{h-constant}
Let $X$ be a smooth complex projective surface with non-negative Kodaira dimension, and let ${C=C_1+ \cdots + C_n} \subset X$ be a transversal configuration of smooth curves having $n \geq 2$ irreducible components $C_1, \dots , C_n$. Then, we have
$$K_X.C +4 \sum_{i=1}^n(1-g(C_i)) - t_{2} + \sum_{r \geq 3} (r-4)t_{r} \leq 3c_{2}(X)-K_{X}^{2}.$$
\end{thm}
\begin{proof}
Let $\text{Sing}({C})$ denotes the set of singular points of the configuration. We define $S= \{ P \in {\rm Sing}(C) : {\rm mult}_{p} \geq 3\}$, and denote by $p_{1}, ..., p_{k}$ the points of $S$. Consider the blowing up of $X$ at $S$, namely
\[ \sigma : Y \longrightarrow X.\]
Following \cite[Section~2.4]{Miyaoka1}, we set $\tilde{C}:= \tilde{C}_1 + \cdots + \tilde{C}_n$, $\tilde{C_i}$ being the strict transform of $C_i$ under $\sigma$. The idea is to use the Bogomolov-Miyaoka-Yau inequality
\[3c_2(Y) - 3e(\tilde{C}) \geq (K_Y + \tilde{C})^2,\]
and thus we now need to compute the terms in the above inequality. We see that
\begin{align*}
& c_2(Y) = c_2(X) +k,\\
& e(\tilde{C}) = \sum_{i=1}^n(2-2g(C_i)) - t_2,
\end{align*}
which yield $c_2(Y) - e(\tilde{C}) = c_2(X)+k-\sum_{i=1}^n(2-2g(C_i))+t_2$.

Notice that, since $\kappa(X) \geq 0$, we have that $\vert mK_X \vert \neq \emptyset$ for some $m\geq 1$, and thus $mK_X$ is linearly equivalent to an effective divisor $D$. Therefore,
\[ K_Y + \tilde{C} = (\sigma^*K_X + E) + \tilde{C} = \sigma^*K_X + (E+\tilde{C}) = \frac{1}{m}\sigma^*D+ (E+\tilde{C}),\]
where $E:=\sum_{j=1}^k E_j$ is the sum of all exceptional divisors. It follows that $K_Y + \tilde{C}$ is numerically equivalent to a rational effective divisor, which in turn allows us to use the Bogomolov-Miyaoka-Yau inequality according to \cite[Corollary 1.2]{Miyaoka1}.
Now we have $K_Y+\tilde{C}  = \sigma^*(K_X +C) - \sum_{j=1}^k (m_j-1)E_j$ and
$$(K_Y+\tilde{C})^2 = K_{X}^{2} + K_{X} \cdot C + \sum_{i=1}^n(2g(C_i)-2) + 2\sum_{i<j}C_{i} . C_{j} - \sum_{j=1}^{k}(m_{j}-1)^2 =$$
$$K_{X}^{2} + K_{X}.C + \sum_{i=1}^n(2g(C_i)-2) + 2t_{2} + \sum_{j=1}^{k}(m_{j}-1),
$$
where the first equality is obtained using that $2g(C) - 2 = C . (C + K_{X}) = \sum_{i=1}^n(2g(C_i)-2) + 2 \sum_{i<j}C_{i}.C_{j}$ and the second equality is explained in \cite[Theorem 2.1]{RP2015}.
By the Bogomolov-Miyaoka-Yau inequality, we see that
$$K_{X}^{2} + K_{X}.C + \sum_{i=1}^n(2g(C_i)-2) + 2t_{2} + \sum_{j=1}^{k}(m_{j}-1) \leq 3\Big(c_2(X)+k-\sum_{i=1}^n(2-2g(C_i))+t_{2} \Big)$$ and  finally
$$K_{X}^{2}+K_{X}.C + 4\sum_{i=1}^n(1-g(C_i)) - t_{2} + \sum_{r\geq 3} (r-4)t_r \leq 3c_{2}(X),$$
which completes the proof.
\end{proof}

Now we can prove the first main result of this paper.
\begin{thm}\label{bound}
In the setting of the previous theorem, one has
$$h(X; {C}) \geq -4 + \frac{K_{X}^{2} -3c_{2}(X)+ \sum_{i=1}^n(2-2g(C_i))+t_{2}}{s}$$
\end{thm}
\begin{proof}
Let $\tilde{C}$ be the strict transform of $C = C_1 + \cdots + C_n$ in the blow-up at the $s$ singular points of the configuration. We observe that
\[\tilde{C}^2/s = \frac{C^2 - \sum_{r\geq 2} r^{2}t_{r}}{s} = \frac{\sum_{j=1}^{n}C_{j}^{2} + I_d - \sum_{r \geq 2} r^2t_r}{s},\]
where $I_d := 2 \sum_{i<j} C_i . C_j$ is the number of incidences of the collection ${C}$ of smooth curves on $X$ with pairwise transversal intersections. Obviously one has
\[I_d - \sum_{r \geq 2} r^2t_r = - \sum_{r \geq 2} rt_r,\]
and moreover, by arguing in a similar way, we can rephrase the bound in Theorem \ref{h-constant} in the following way:
\[-\sum_{r \geq 2}r t_r \geq K_{X}^{2} + K_{X}.C + 4\sum_{i=1}^n(1-g(C_i)) + t_{2}-4s-3c_{2}(X).\]
This yields
\[  h(X;{C})   \geq -4 + \frac{K_{X}^{2} -3c_{2}(X)+ 2\sum_{i=1}^n(1-g(C_i))+t_{2}}{s},\]
and we are done.
\end{proof}


We are now going to give a couple of examples.

\begin{ex}
If $A$ is a smooth abelian surface and $\mathcal{C}$ is a configuration of elliptic curves. Let us briefly recall that such a configuration has transversal intersection points only. Indeed, every holomorphic map between abelian varieties is the composite of a group homomorphisms followed by a translation \cite[Proposition 1.2.1]{LangeB}. If $f:E\rightarrow A$ is a holomorphic map from an elliptic curve into $A$, there exists a unique $g: \C \rightarrow\C^2$ such that the homomorphism $\bar{g}: E \rightarrow A$ induced by $g$ yields a factorization $f = t_{f(0)} \circ \bar{g}$, where $t_{f(0)}$ is the translation by $f(0)$ in $A$. Therefore, every elliptic curve is the image of a line in $\C^2$, and the claim follows.
By Theorem \ref{h-constant} for $A$ and $\mathcal{C}$ as above, we obtain
$$t_{2} + t_{3} \geq \sum_{r \geq 5}(r-4)t_{r}.$$
Moreover, we have the following bound
$$h(A;\mathcal{C}) \geq -4 + \frac{t_2}{s} \geq - 4,$$
and as it is pointed out in \cite{Roulleau} this bound is sharp. To this end, consider an abelian surface equipped with the complex multiplication given by $e^{2 \pi i /3}$. Then the equality is provided by the following configuration $$\mathcal{E} = F_{1} + F_{2} + \triangle + \Gamma,$$
where $F_{1},F_{2}$ are fibers, $\triangle$ is the diagonal, and $\Gamma$ is the graph of the complex multiplication.
\end{ex}

\begin{ex}
Let $\mathcal{L}$ be a configuration of curves, all with the same genus $g$, such that the following conditions are satisfied:
\begin{enumerate}
\item all intersection points are pairwise transversal,
\item there exists a positive integer $k$ such that for any pair $C_{1},C_{2} \in \mathcal{L}$ (not necessarily distinct) we have $C_{1}.C_{2} = k \geq 1$,
\item there is no point such that all curves meet.
\end{enumerate}
Then $\mathcal{L}$ is said to be $k$-regular. Urz\'ua \cite[Remark 7.4]{Urzua} has shown that if $\mathcal{L} \subset X$ is $k$-regular and consists of $n\geq 3$ curves, then $s \geq n$. In this case, our bound in Theorem \ref{bound} depends only on genus of the curves, and it is independent of the number of curves in the configuration. 

First of all, by Theorem \ref{bound}, 
\[h(X; \mathcal{L})   \geq -4 + \frac{K_{X}^{2} - 3c_{2}(X) + 2n(1-g) + t_{2}}{s} \] \[ \geq -4 + \frac{K_{X}^{2} - 3c_{2}(X) + 2n(1-g)}{s}. \]
Since $K_{X}^{2} -3c_{2}(X) \leq 0$ by the Miyaoka-Yau inequality, then, in order to find a lower bound on $h(X;\mathcal{L})$, we need to find a lower bound on $s$. Clearly $s \geq n$, hence
$$h(X; \mathcal{L}) \geq -4 - 2(g-1) - (3c_{2}(X) - K_{X}^{2}).$$
\end{ex}

\begin{rmk}
	Similarly to the case of surfaces with numerically trivial canonical class treated in \cite{RP2015}, we can obtain a uniform bound on the H-constant of rational curves on a surface $X$ of non-negative Kodaira dimension. In fact, a similar argument to \cite[Corollary 2.4]{RP2015} shows that 
    \[ H_{\text{rational}}(X) \geq -4- \delta(X), \]
    where $\delta(X):=3e(X)-K_{X}^{2}$. Indeed, this is not optimal, as one immediately sees by comparing with \cite[Corollary 2.4]{RP2015}. 

It is somewhat curious how we can obtain information on the geometry of configuration by looking at the H-constant. Suppose we are given a \textit{star configuration}, i.e.~a (connected) configuration $C=C_1+ \cdots + C_n$ of rational curves intersecting at a single point $p$ and nowhere else. Then, we have that
\[ -4-\delta(X) + 2n \leq h(C) = -3n - K_X.C, \]
which yields the inequality
\[ 5n + K_X.C \leq 4 + \delta(X). \]
This inequality governs the geometry of star configurations: for instance, if $X$ is a K3 surface, then such a configuration consists of at most 15 rational curves, while if $X$ is an Enriques surface there can only be at most 8 such curves.

It is also possible to obtain a uniform bound on the H-constants of configurations of elliptic curves. It is readily seen that
\[ H_{\text{elliptic}}(X) \geq -4-\delta(X). \]
Unfortunately, this technique does not deliver any such information on configurations of curves of higher genus.
\end{rmk}


\section{H-constants for line configurations on smooth complete intersection surfaces}\label{3}
 
In \cite{BdRHHLPSz}, the authors have shown that if $\mathcal{L}$ is a line configuration on the complex projective plane, then $h(\mathbb{P}^{2}_{\mathbb{C}}; \mathcal{L}) \geq -4$. In particular, this result proves the existence of a uniform bound on the H-constants of line configurations on $\mathbb{P}^2_\C$. In this section, we will prove that there exists a uniform bound on Harbourne constants for line configurations on smooth hypersurfaces of degree $d$ in $\mathbb{P}^3_\C$ for any $d \geq 4$. It is worth mentioning that the first non-trivial case $d=3$ was fully described in \cite[Proposition 4.1]{Pokora}.

Our interest in this direction arises because of the following observation. In \cite{RP2015}, we studied the local negativity phenomenon for transversal configurations of smooth rational curves on surfaces with numerically trivial canonical class: we provide a bound for their H-constants and we show this bound to be quite accurate in many examples (K3 surfaces or Enriques surfaces). However, even in the easy case of line arrangements of quartics surfaces in $\IP^3_\C$ (which are K3 surfaces), this estimate turns out to be not as accurate. In fact, let $S \subset \mathbb{P}^3_\C$ be the Schur quartic surface, namely the hypersurface of $\mathbb{P}^3_\C$ given by the following equation:
\[ S_{4} : \qquad x^4 - xy^3 = z^4 - zw^3.\]
Let us consider a configuration $L$ consisting of four lines meeting at a unique point (see, for example, \cite[Example 3.5]{RP2015}). By Theorem \ref{bound}, we conclude that $h(S_{4};L) \geq -68$, while a direct computation shows that $h(S_{4};L) = -12$. This observation motivated us in pursuing a detailed study of configurations of lines on complete intersection surfaces $X \subset \mathbb{P}^{n+2}_\C$. As a further motivation, let us also mention here that, up to now, the most negative values of H-constants were obtained using line configurations on smooth hypersurfaces. Our analysis is divided into two parts: in order to show the main ideas, we first deal with the case of hypersurfaces in $\IP^3$. Afterwards, we will see how the general case follows from the hypersurface case.\\

To start with, we need to define connected configurations of lines. Let $X$ be a smooth hypersurface of degree $d \geq 4$, and let $\mathcal{L}$ be a configuration of lines on $X$. We say that $\mathcal{L}$ is \emph{connected} if the union of the lines in $\mathcal{L}$ is connected as a subset of $\mathbb{P}^3_\C$. Let us also observe that if $\mathcal L$ is a (connected) configuration of $k \geq 2$ lines on $X$, then
\[
 h(X; \mathcal{L}) = -\frac {(d-2)k + \sum r t_r}{\sum t_r}.
\]
We move on to stating the main result for hypersurfaces.

\begin{thm}\label{h-const_lines}
Let $X$ be a surface in $\mathbb{P}^3_\C$ of degree $d\geq 4$, and let $\mathcal{L}$ be a connected configuration of $k \geq 2$ lines on $X$. Suppose that $m \in \mathbb{N}$ is the positive integer such that $t_m \neq 0$ and $t_r = 0$ for all $r > m$. Then,
\[h(X;\mathcal{L}) \geq -m(d-1).\]
Moreover, this bound is sharp, and it is achieved by the configuration consisting of $m$ lines meeting at a single point.
\end{thm}

This result has to be interpreted in the following sense: a low H-constant corresponds to a point of high multiplicity in our configuration. As a consequence, we will obtain the bound on the H-constants that we seek:

\begin{thm}\label{degree_d}
For a connected configuration $\mathcal{L}$ of $k \geq 2$ lines on $X$ of degree $d \geq 4$ one has
\[ h(X;\mathcal{L}) \geq -d(d-1).\]
\end{thm}

\begin{proof}[Proof of Theorem \ref{h-const_lines}]
Recall that given a line $\ell \subset X$, $\ell^2 = 2-d$. Let $P$ be an $m$-point of $\mathcal{L}$. Since $X$ is smooth, the $m$ lines meeting at $P$ must be coplanar. Let us call these lines $\ell_1, \dots, \ell_m$, and consider the plane $\Pi$ containing them. The hyperplane section $X \cap \Pi$ can be written as
\[ X \cap \Pi = \ell_1 \cup \dots \cup \ell_m \cup \ell_1' \cup \dots \cup \ell_e' \cup C_1 \cup \dots \cup C_s,\]
where $\ell'_j$ is a line not through $P$, for all $0 \leq j \leq e$, $C_i$ is a plane curve of degree $\deg(C_i) \geq 2$, for all $0 \leq i \leq s $, and 
\[ m+e +\sum_{i=1}^s \deg(C_i) = d.\]
After possibly renumbering, the lines $\ell_1', \dots, \ell_{\bar{e}}'$ belong to the configuration $\mathcal{L}$ ($0 \leq \bar{e} \leq e$), while the lines $\ell_{\bar{e}+1}, \dots , \ell_e$ do not. Set
\[\mathcal{M}:= \lbrace \ell_1, \dots, \ell_m, \ell_1', \dots, \ell_{\bar{e}}' \rbrace,\]
and $\mathcal{N}:= \mathcal{L}\setminus\mathcal{M}$. Finally, letting $k\geq 2$ be the number of lines in $\mathcal{L}$, put $n:=k-m-\bar{e}$. Our aim is to find a non-negative integer $b$ such that $h(\mathcal{L}) \geq -b$. This quantity will depend on the degree $d$ and the complexity of our configuration, namely the maximum number of lines in $\mathcal{L}$ intersecting at a single point.

Notice that
\begin{align*}
h(X;\mathcal{L}) &= - \frac{(d-2)k + \sum rt_r}{\sum t_r} = - \frac{(d-2)(n+m+\bar{e}) + \sum rt_r}{\sum t_r},
\end{align*}
and assume that $h(X;\mathcal{L}) \geq - b$. This condition is equivalent to 
\begin{align}\label{eq1}
(d-2)(n+m+\bar{e}) + \sum (r-b)t_r \leq 0.
\end{align}
Let $M$ (respectively $N$) be the divisor associated to the configuration $\mathcal{M}$ (respectively $\mathcal{N}$). We will now proceed with splitting the summations according to whether a $r$-point lies on $M \setminus N$, $M \cap N$ or $N \setminus M$. To this end, for a set $S$, let us denote by $t_{r,S}$ the number of $r$-points that lie in $S$. Then, one has:
\begin{align*}
\sum r t_r &= \sum r t_{r,M \setminus N} + \sum r t_{r, M \cap N} + \sum r t_{r, N \setminus M};\\
\sum t_r &= \sum t_{r,M \setminus N} + \sum t_{r, M \cap N} + \sum t_{r, N \setminus M}.
\end{align*}
Now observe that we can write $n$ as
\begin{align}\label{eq3}
n = n_R + \sum (r-1) t_{r,M \cap N},
\end{align}
where $n_R$ is the number of lines in $\mathcal{N}$ which do not intersect any line in $\mathcal{M}$. This follows because, if $Q$ is a point where a line in $\mathcal{M}$ and a line in $\mathcal{N}$ intersect, then there is no other line in $\mathcal{M}$ going through $Q$ (otherwise $Q$ would be singular). Moreover, we can bound $n_R$ as 
\begin{align}\label{eq4}
n_R \leq \sum r t_{r,N \setminus M}.
\end{align}
Indeed, as the configuration is connected, on each line contributing to $n_R$ there is an $r$-point in $N \setminus M$, for some $r \geq 2$, and taking the weighted sum yields the estimate.

Now, notice that the quantities $t_{r,M \setminus N}$ only depend on the configuration $M$, which is a configuration of lines in $\mathbb{P}^2_\C$. We aim at giving an upper bound for the number of lines of the configuration ($m+\bar{e}$ in our case) in terms of the $t_r$'s. In general, letting $\mathcal{C}$ be a configuration of $c \geq 2$ lines in $\mathbb{P}^2_\C$, and $\ell$ a line in $\mathcal{C}$, one has:
\begin{align}
c = 1 + \sum (r-1)t_{r,\ell}.
\end{align}
Moreover, unless $\mathcal{C}$ is the configuration of $c$ lines meeting at one point,
\begin{align}\label{eq2} 
c = 1 + \sum (r-1)t_{r,\ell} \leq \sum rt_r,
\end{align}
as there exists at least an $r$-point outside of $\ell$, for some $r \geq 2$. Nevertheless, the same estimate holds even in the case of all lines meeting at a single point. Then, we can apply (\ref{eq3}), (\ref{eq4}) and (\ref{eq2}) to the left-hand side of (\ref{eq1}) to get:
\begin{align*}
(d-2)&(n+m+\bar{e}) + \sum (r-b)t_r \\
& \leq \sum \big[(d-1)r-(d-2+b) \big]t_{r,M \cap N}\\
 &\quad + \sum \big[(d-1)r-b \big](t_{r,M \setminus N} + t_{r,N \setminus M}).
\end{align*}
This shows that it is enough to set $b := m(d-1)$ to satisfy condition (\ref{eq1}), and this completes the proof of the existence of the bound. 

We are left to show that this bound is actually sharp. The condition $h(X;\mathcal{L}) = -m(d-1)$ is equivalent to 
\begin{align*}
& \sum \big[(d-1)(r-m-1)+1\big]t_{r,M \cap N} \\
& \quad + \sum \big[(d-1)(r-m) \big](t_{r,M \setminus N} + t_{r,N \setminus M})=0
\end{align*}
As the first summation is always negative and the second is non-positive, it follows that $M \cap N = \emptyset$. By the connectedness of $\mathcal{L}$, this in turn implies that $\mathcal{N} = \emptyset$, and thus
\[ \sum \big[(d-1)(r-m) \big]t_{r,M}=0.\]
This equality implies that $\mathcal{L} = \mathcal{M}$ must be the configuration of $m$ lines meeting at a single point.

\end{proof}

Let us present the following example, which shows that the bound in Theorem \ref{h-const_lines} is sharp.

\begin{ex}(Schur degree-$d$ surface)
For a positive integer $d$ let us consider the following surface
\[S_d: \ x^d - xy^{d-1} = z^d - zw^{d-1}.\]
We will call this surfaces as the Schur degree-$d$ surface. Observe that for $d=4$ we recover the celebrated Schur quartic surface. By intersecting $X$ with the hyperplane $\lbrace x = 0 \rbrace$ one obtains a configuration $\mathcal{L}$ of $d$ lines $\ell_1, \dots , \ell_d$ intersecting at one $d$-fold point. For $m \leq d$, consider a subconfiguration $\mathcal{L}_m$ of $\mathcal{L}$ consisting of $m$ lines. By the adjuction formula we can compute that each line has self-intersection $\ell_i ^2 = 2-d$. Therefore,
\[h(S_d; \mathcal{L}_m)= -m(d-1).\]
In case $m=d$, we notice that, as the degree becomes larger, the Harbourne constant gets more negative, i.e.
\[ \lim_{d \to + \infty} h(S_d; \mathcal{L}_d) = - \infty.\]
This phenomenon was already observed in \cite[Example 3.4]{Pokora}. However, in that example one has that the H-constant decreases like $-d/2$ as $d \rightarrow +\infty$, whereas for the degree $d$-Schur surface we have
\[h(S_d; \mathcal{L}_d) \sim -d^2\]
as $d \rightarrow +\infty$, hence quadratic growth for the negativity.
\end{ex}

So far, we have dealt with connected configurations of lines only. We would like to obtain a bound which holds for an arbitrary configuration of lines, possibly non-connected. The first step toward such a result is showing that the bound $h(X; \mathcal{L})\geq -d(d-1)$ holds for a configuration $\mathcal{L}$ whose connected components consist of at least two lines.

\begin{lemma}\label{manyconnected}
Let $X$ be a surface of degree $d\geq 4$ in $\mathbb{P}^3_\C$, and let $\mathcal{L}= \mathcal{L}_1 \sqcup \dots \sqcup \mathcal{L}_n$  ($n \geq 1$) be a configuration of lines whose connected components are the $\mathcal{L}_i$'s and such that each $\mathcal{L}_i$ consists of at least two lines. Then
\[ h(X; \mathcal{L}) \geq -d(d-1).\]
\end{lemma}

\proof
Set $s_i := \#{\rm Sing}(\mathcal{L}_i)$ ($1 \leq i \leq n$), and $s := \sum_{i=1}^n s_i$. Let $L$ be the divisor associated to $\mathcal{L}$, and let $L_i$ be the one associated to $\mathcal{L}_i$ ($i = 1, \dots , n$). The existence of at least two lines in each $\mathcal{L}_i$ guarantees that $s_i >0$ ($1 \leq i \leq n$). Then, by means of Theorem \ref{degree_d}, we get
\begin{align*}
h(X;\mathcal{L}) &=  \frac{\sum_{C \in \mathcal{L}}C^2 - \sum rt_{r,L}}{s} \\
                &= \sum_{i=1}^n \frac{\sum_{C \in \mathcal{L}_i} C^2 - \sum rt_{r,L_i}}{s}\\
                &= \frac{1}{s}\sum_{i=1}^n s_i\frac{\sum_{C \in\mathcal{L}_i} C^2 - \sum rt_{r,L_i}}{s_i}\\
                &\geq -\frac{1}{s} \sum_{i=1}^n d(d-1) s_i = -d(d-1).
\end{align*}
\endproof
 
We are left to handle the case in which some of the $\mathcal{L}_i$'s consist of an isolated line (and thus $s_i = 0$). These isolated lines give non-trivial contribution to the H-constant.

\begin{thm}\label{arbitraryconfig}
Let $X$ be a surface of degree $d\geq 4$ in $\mathbb{P}^3_\C$, and let
$$\mathcal{L}= \mathcal{L}_1 \sqcup \dots \sqcup \mathcal{L}_m \sqcup \mathcal{L}_{m+1} \sqcup \dots \sqcup \mathcal{L}_{m+n} \qquad (m \geq 1, \, n \geq 0)$$
be a configuration of lines whose connected components are the $\mathcal{L}_i$'s, and such that $\mathcal{L}_i$ consists of at least two lines for $i = 1, \dots , m$, and $\mathcal{L}_i$ is an isolated line for $i = m+1, \dots, m+n$. Then,
\[ h(X; \mathcal{L}) \geq -d(d-1)+ (2-d)n.\]
\end{thm}

\proof
The assumptions imply that $s_i >0$ for $i= 1, \dots , m$, and $s_i = 0$ otherwise. By means of Lemma \ref{manyconnected}, we have
\begin{align*}
h(X;\mathcal{L}) &= \frac{\sum_\mathcal{L} C^2 - \sum_\mathcal{L} rt_r}{s_1 + \cdots + s_m} \\
                &= \sum_{i=1}^m \frac{\sum_{\mathcal{L}_i} C^2 - \sum_{\mathcal{L}_i} rt_r}{s_1 + \cdots + s_m} +  \sum_{i=m+1}^{m+n} \frac{\sum_{\mathcal{L}_i} C^2 - \sum_{\mathcal{L}_i} rt_r}{s_1 + \cdots + s_m}\\
                &\geq -d(d-1) +  \sum_{i=m+1}^{m+n} \frac{\sum_{\mathcal{L}_i} C^2 - \sum_{\mathcal{L}_i} rt_r}{s_1 + \cdots + s_m}\\
                &= -d(d-1) + \sum_{m+1}^{m+n} \frac{2-d}{s_1 + \cdots + s_m} \geq -d(d-1) + (2-d)n.
\end{align*}
\endproof

At a closer look, our proofs so far in this section are valid also for $d=3$. The next result gives a uniform bound for the Harbourne constant of line configurations, recovering the same situation as for lines in $\mathbb{P}^2_\C$. Here, we are invoking results of Miyaoka \cite{Miyaoka1} on surfaces with nonnegative Kodaira dimension, therefore we do need the condition $d \geq 4$.

\begin{cor}
Let $X$ be a smooth surface of degree $d\geq 4$ in $\mathbb{P}^3$, and let $\mathcal{L}$ be a configuration of lines on $X$. Then,
\[h(X; \mathcal{L}) \geq -2d^3 + 7d^2 - 6d -2.\]
\end{cor}
\proof
A result of Miyaoka \cite{Miyaoka1} states that the maximum number of disjoint lines on a hypersurface in $\mathbb{P}^3$ of degree $d\geq 4$ is $2d(d-2)$. Let
$$\mathcal{L}= \mathcal{L}_1 \sqcup \dots \sqcup \mathcal{L}_m \sqcup \mathcal{L}_{m+1} \sqcup \dots \sqcup \mathcal{L}_{m+n} \qquad (m \geq 1, \, n \geq 0)$$
be an arbitrary configuration of lines. As $m\geq 1$, $n \leq 2d(d-2)-1$: if $n = 2d(d-2)$, as $m \geq 1$, there would be $2d(d-2)+1$ disjoint lines, which is clearly impossible. Then, by means of Theorem \ref{arbitraryconfig}, we have
\[h(X;\mathcal{L}) \geq -d(d-1) + (2-d)n \geq -2d^3 + 7d^2 - 6d -2.\]
\endproof

Now, we would like to discuss the case of complete intersection surfaces in projective spaces. We illustrate how the argument in the hypersurface case generalizes to complete intersections in the case of connected configurations, leaving the case of possibly non-connected configurations to the interested reader. 

Let $X$ be a smooth complete intersection of multi-degree $(d_1,\dots,d_n)$ in $\IP^{n+2}$, and let $\kl$ be a (connected) configuration of $k \geq 2$ lines on $X$. As an embedded surface, $X$ has degree $\prod_{i=1}^n d_i$, and adjuction formula takes the form
\[ K_X = \Big(-n-3+\sum_{i=1}^n d_i \Big) \xi, \]
$\xi$ being the class of a hyperplane section. If $\ell \subset X$ is a line, then $\ell^2 = n+1 - \sum_{i=1}^n d_i$. If $\sum_{i=1}^n d_i \leq n+1$, then we can find an easy bound on the H-constant of a configuration $\kl$ of lines on $X$: letting $m$ be the positive integer such that $t_m \neq 0$ and $t_r =0$ for all $r >m$, one has $h(\kl) \geq -m$.\\

Therefore, from now on, we will always assume that $\sum_{i=1}^n d_i > n+1$. Letting $m$ be the integer described above, let $p$ be an $m$-fold point of $\kl$ (i.e.~a point where exactly $m$ lines of $\kl$ meet). As $X$ is smooth, the lines concurring at $p$ must lie on the same 2-plane $\Pi$. Consider the intersection 
\[ X \cap \Pi = \bigcup_{i=1}^m \ell_i \cup \bigcup_{j=1}^e \ell'_j \cup \bigcup_{k=1}^s C_k \cup \bigcup_{l=1}^t P_l, \]
consisting of the $m$ lines through $p$, further $e$ lines $\ell_j'$ (not passing through $p$ or not in $\kl$), some curves $C_k$ of degree at least $2$, and some isolated points $P_l$. After possibly reordering, we can assume that $\ell'_j$ is a line in $\kl$ which does not go through $p$ for $j=1, \dots, \bar{e}$, while $\ell'_j$ does not belong to $\kl$ for $j=\bar{e}+1, \dots,e$. Now we define $\km = \lbrace {\ell_1, \dots, \ell_m,\ell'_1, \dots, \ell_{\bar{e}}} \rbrace$ and $\kn := \kl \setminus \km$, and set $n:= \# \kn$.\\

Let $n_R$ be the number of lines of $\kl$ intersecting one of the $\ell'_j$'s for $j=\bar{e}+1, \dots, e$, one of the $C_k$'s, or passing through one of the $P_l$'s. Following the proof of Theorem \ref{h-const_lines}, one has that:
\begin{enumerate}
	\item $n=n_R + \sum_r (r-1)t_{r,M \cap N}$;
	\item $n_R \leq \sum_r rt_{r,N \setminus M}$;
	\item $m + \bar{e} \leq \sum_r r t_{r,M \setminus N}$.
\end{enumerate}
These facts imply the following:

\begin{thm}
	In the above hypothesis, $h(X;\kl) \geq -m \big(\sum_{i=1}^n d_i - n \big)$. The bound is sharp, and it is attained for a configuration of $m$ lines interesecting at a single point.
\end{thm}

We can also obtain a bound on the H-constants that only depends on multi-degree $(d_1, \dots, d_n)$ of $X$.

\begin{thm}\label{complete_intersection}
	For a smooth complete intersection $X \subset \IP^{n+2}_\C$ of multi-degree $(d_1, \dots, d_n)$ and a configuration $\kl$ of lines on $X$, we have the following bound on the H-constants:
	\[h(X;\kl)\geq -\Big(\max_{1 \leq i \leq n} d_i \Big) \cdot \Bigg( \sum_{i=1}^n d_i - n \Bigg).\]
\end{thm}	

\begin{proof}
Let $m$ be the positive integer such that $t_m \neq 0$ and $t_r = 0$ for $r >m$. If $\ell$ is a general line on $\Pi$, we can use $\ell$ to determine the degree of the (reducible) plane curve in $X \cap \Pi$ (we disregard the isolated points in this intersection). In fact, this degree is at most $\max_{1 \leq i \leq n} d_i$. Indeed, writing $X = \bigcap_{i=1}^n Z_i$, with $\deg Z_i = d_i$ ($i=1, \dots, n$), there exists an index $\alpha$ such that $\ell \nsubseteq Z_\alpha$ (otherwise $\ell \subset X$). For such $\alpha$, $\ell \cap Z_\alpha$ consists of $d_\alpha$ points (here we use the genericity assumption), so that $\ell \cap X$ consists of at most $d_\alpha$ points. Taking the maximum of the $d_i's$, we have shown that the degree of the plane curve in $X \cap \Pi$ is at most $\max_{1 \leq i \leq n} d_i$, hence also that $m \leq \max_{1 \leq i \leq n} d_i$.
\end{proof}

\paragraph*{\emph{Acknowledgement.}}
It is our pleasure to thank Davide Cesare Veniani for fruitful conversations, and especially for showing us a proof of Theorem \ref{degree_d} in the case $d=4$, from which our argument grew out. We also would like to thank Xavier Roulleau for useful remarks on an earlier draft of this manuscript. The second author was partially supported by National Science Centre Poland Grant 2014/15/N/ST1/02102 and the major part of the project was conducted when he was a member of SFB $45$ \emph{"Periods, moduli spaces and arithmetic of algebraic varieties"}.
Finally, we would like to thank the anonymous referee for very valuable comments and remarks which allowed to improve the paper.



\end{document}